\UseRawInputEncoding 
\documentclass[a4paper,11pt]{article}
\pdfoutput=1

\usepackage[active]{srcltx}
\usepackage{verbatim}
\usepackage{epsfig,graphicx,color,mathrsfs}
\usepackage{graphicx}
\usepackage{amsmath,amssymb,amsthm,amsfonts}
\usepackage{amssymb}
\usepackage{mathtools,xparse}
\usepackage[english]{babel}
\usepackage[usenames,dvipsnames,svgnames,table]{xcolor}
\usepackage{fullpage}
\usepackage{tikz}
\usepackage{caption}
\usepackage{float}
\usepackage{cite}
\usepackage{manyfoot}
\usepackage[all, color]{xypic}
\usepackage{tikz-cd}

\usepackage{mathtools}
\mathtoolsset{showonlyrefs=true}

\usepackage[margin=1.5in]{geometry}

\usepackage[unicode,bookmarks, pdftex, backref = page]{hyperref}
\definecolor{newblue}{RGB}{0,102,204}
\definecolor{newred}{RGB}{206,32,41}
\hypersetup{colorlinks=true,citecolor=newblue,linkcolor=newred,
  urlcolor=magenta,
pdfpagemode=UseNone, breaklinks=true}

\usepackage{palatino,hhline, bbm}

\newtheorem{theorem}{Theorem}[section]
\newtheorem{proposition}[theorem]{Proposition}
\newtheorem{lemma}[theorem]{Lemma}
\newtheorem{corollary}[theorem]{Corollary}
\newtheorem{definition}[theorem]{Definition}

\newtheorem*{teoA}{Theorem A}
\newtheorem*{teoB}{Theorem B}

\newtheorem*{Op1}{Problem 1}
\newtheorem*{Op2}{Problem 2}

\theoremstyle{definition}
\newtheorem{remark}[theorem]{Remark}
\newtheorem{example}[theorem]{Example}
\newtheorem{assumption}[theorem]{Assumption}

\renewcommand{\to}{\longrightarrow}

\newcommand{\PP}{\mathbb{P}}
\newcommand{\mO}{\mathcal{O}}

\usepackage{epigraph}
\title{Pluri-cotangent maps of surfaces of general type}
\author{Francesco Polizzi, Xavier Roulleau}
\date{}
\begin{document}
\maketitle
%



\begin{abstract}
Let $X$ be a compact, complex surface of general type whose cotangent bundle $\Omega_X$ is strongly semi-ample. We study the pluri-cotangent maps of $X$, namely the morphisms 
$\psi_n \colon \PP(\Omega_X) \to \PP(H^0(X, \, S^n \Omega_X))$ defined by the vector space of global sections $H^0(X, \, S^n \Omega_X)$.

\end{abstract}


\Footnotetext{{}}{2010 \textit{Mathematics Subject Classification}: 14J29}

\Footnotetext{{}} {\textit{Keywords}: surfaces of general type, pluri-cotangent maps}


\setcounter{section}{-1}

\section{Introduction} \label{sec:intro}

Let $X$ be a complex surface of general type and assume that its cotangent bundle $\Omega_X$ is strongly semi-ample. This means that for some integer $n \geq 1$ the symmetric power $S^n \Omega_X$ is globally generated, namely,  the evaluation map 
\begin{equation} \label{eq:surjectivity-1-intro}
H^0(X, \, S^n \Omega_X) \otimes \mathcal{O}_X \longrightarrow S^n \Omega_X
\end{equation}
is surjective (this condition implies in particular that $X$ is minimal and $K_X$ is ample). Recalling that we have a natural identification between $H^0(X, \, S^n \Omega_X)$ and $H^0(\mathbb{P}(\Omega_X), \, \mathcal{O}_{\mathbb{P}(\Omega_X)}(n))$, from the surjectivity of \eqref{eq:surjectivity-1-intro} we infer that the induced evaluation map
\begin{equation} \label{eq:surjectivity 2-intro}
H^0(X, \, S^n \Omega_X) \otimes \mathcal{O}_{\mathbb{P}(\Omega_X)} \longrightarrow \mathcal{O}_{\mathbb{P}(\Omega_X)}(n)
\end{equation}
is also surjective, and so it defines a morphism
\begin{equation} \label{eq:cotang-intro}
\psi_n \colon \mathbb{P}(\Omega_X) \to \mathbb{P}(H^0(X, \, S^n \Omega_X)),
 \end{equation}
that we call the $n$\emph{th pluri-cotangent map} of $X$.  The case $n=1$ was studied by the second author in \cite{Rou09}: 
it turns out that, as soon as $\Omega_X$ is globally generated and $q(X) >3 $, the cotangent map $\psi_1 \colon \mathbb{P}(\Omega_X) \to \mathbb{P}(H^0(X, \, \Omega_X)) \simeq \mathbb{P}^{q(X)-1}$ is a generically finite morphism onto its image. 
 
In this note we generalize this result to the case $n \geq 2$, giving conditions for the generic finiteness of the pluri-cotangent maps $ \psi_n$. Let us explain how the paper is organized, and what are the outcomes of our investigation. 

\smallskip
In Section \ref{sec:pluri}  we put the problem into context and we collect some preliminary facts  which will be needed  in the sequel of the manuscript. 

\smallskip
The first result we show (in Section \ref{sec:image pluri-cotangent}) is

\begin{teoA}[see Theorem \ref{thm:image-n-cotangent}]
Let $n \geq 2$ be such that $S^n \Omega_X$ is globally generated. If $h^0(X, \, S^n \Omega_X) > \frac{1}{2}(n+1)(n+2)$ then $\psi_n$ is generically finite onto its image. In this case, the exceptional locus $\mathrm{exc}(\psi_n)$ is a Zariski-closed, possibly empty subset of $ \PP (H^0(X, \, S^n \Omega_X))$ of dimension at most $1$. 
\end{teoA}
The proof of Theorem A is obtained by generalizing the geometrical arguments used in \cite{Rou09}. We also exploit some results contained in the recent paper \cite{MU19} by Mistretta and Urbinati, allowing us to prove the finiteness of the $n$th Gauss map of $X$ (Proposition \ref{prop:gauss-finite}), together with the description of surfaces $2$-covered by curves of degree $n$ (Proposition \ref{prop:n-covered}), classically obtained by Bompiani \cite{Bom21} and rediscovered with modern techniques by Pirio and Russo \cite{PR13}.

\smallskip
In Section \ref{sec:generically-finiteness} we prove the following apparently simple fact which, however,  we have not been able to find in the literature.

\begin{teoB}[see Theorem \ref{thm:generically-finite-1}]
Let $n \geq 3$  be an integer such that $S^n \Omega_X$ is globally generated and  $\chi(X, \,S^n \Omega_X)  \geq 0$. Then the pluri-cotangent map  $\psi_n \colon \mathbb{P}(\Omega_X) \to \mathbb{P}(H^0(X, \, S^n\Omega_X))$ is generically finite onto its image. 
\end{teoB}

The proof of Theorem B uses the explicit computation of $\chi(X, \, S^n \Omega_X)$, provided in  Lemma \ref{lem:chi-S^n}, in order to show that the inequality $\chi(X, \, S^n \Omega_X) \geq 0$ gives $c_1^2-c_2 >0$. This,  together with Bogomolov's cohomological vanishing (Proposition \ref{prop:semistability of Omega}), implies that $S^n \Omega_X$ is big, so that its  global generation yields the generic finiteness of the pluri-cotangent map (Proposition \ref{prop:c1^2-c2>0}).

\smallskip 
Section \ref{sec:examples} deals with some examples and counterexamples. In Subsections \ref{subsec:complete-intersections} and \ref{subsec:product-quotient} we consider a pair of constructions giving surfaces $X$ such that $\Omega_X$ is neither ample nor globally generated: 
\begin{itemize}
\item[\textbf{(1)}] $X$ is a suitable symmetric complete intersection in an abelian fourfold of the form $A \times E$, where $A$ is an abelian threefold and $E$ is an elliptic curve;
\item[\textbf{(2)}] $X$ is of the form $X=(C \times F)/G$, where $C$ is a smooth hyperelliptic curve of genus $3$, $F$ is a smooth curve of odd genus and $G =\mathbb{Z}_2$ acts with four fixed points on $C$, freely on $F$ and diagonally on the product.
\end{itemize}
In both situations, the vector bundle $S^2 \Omega_X$ turns out to be globally generated, hence $\Omega_X$ is strongly semi-ample, and moreover the pluri-cotangent map $\psi_n$ is generically finite onto its image for all even $n$ (see Propositions \ref{prop:abelian-fourfold},  \ref{prop:cotangent_maps_example}, \ref{prop: S^2 for product-quotient}, \ref{prop:generically-finite-for-all-n}). 

In Subsection \ref{subsec:lambda-6-sharp} we exhibit some counterexamples to Proposition \ref{prop:c1^2-c2>0} in the situation $c_1^2-c_2=0$: they are  smooth ample divisors in abelian threefolds. In fact, if $X$ is such a  divisor, for all $n \geq 1$ we have 
\begin{equation} \label{eq:symmetric-power-intro}
H^0(X, \, S^n \Omega_X)=S^n H^0(X, \, \Omega_X) \simeq 	\mathbb{C}^{\frac{(n+1)(n+2)}{2}},
\end{equation}
and the image $X_n$ of $\psi_n \colon \PP(\Omega_X) \to \PP(H^0(X, \, S^n \Omega_X))$ is projectively equivalent to the $n$th Veronese surface $\nu_n(\PP^2) \subset \PP^{\frac{n(n+3)}{2}}$. Thus, no pluri-cotangent map of $X$ is generically finite onto its image. Under the additional assumption that the image of the Albanese map is smooth, we also show that these counterexamples are the only ones up to finite, \'{e}tale covers (Proposition \ref{prop:all-pluricotangent-dim-2}).

By using finite cyclic covers of a surface $X$ as above, we are also able to construct surfaces of general type  all of whose Gauss maps have arbitrarily large degree, see Remark \ref{rmk:order-Gauss-maps}.

\smallskip
Finally, in Section \ref{sec:open_problems} we state a couple of open problems.

\bigskip
\indent \textbf{Acknowledgments.} This work started in May 2017, when the first author visited the Institut de Math\'ematiques de Marseille. He is grateful to the members of the  \'{e}quipe Analyse, G\'{e}om\'{e}trie et Topologie for the invitation and the hospitality, and to GNSAGA-INdAM for the financial support. He also thanks all the MathOverflow users that generously answered his questions in several threads, see
\medskip

\noindent
\href{https://mathoverflow.net/questions/397682/}{MO397682}, \href{https://mathoverflow.net/questions/412306/}{MO412306}, \href{https://mathoverflow.net/questions/412888/}{MO412888}, \href{https://mathoverflow.net/questions/413988}{MO413988}, \href{https://mathoverflow.net/questions/414382/}{MO414382},  \href{https://mathoverflow.net/questions/414452/}{MO414452},  

\noindent 
\href{https://mathoverflow.net/questions/417972/}{MO417972}, 
\href{https://mathoverflow.net/questions/418607/}{MO418607},
\href{https://mathoverflow.net/questions/430933/}{MO430933},
\href{https://mathoverflow.net/questions/431327/}{MO431327}.

\medskip 
Both authors are grateful to Erwan Rousseau for providing useful references and to Jie Liu, Antonio Rapagnetta and Igor Reider for suggestions and remarks.
\bigskip

\indent \textbf{Notation and conventions.} 
We work over the field $\mathbb{C}$ of complex numbers.
By \emph{surface} we mean a smooth, compact complex surface
$X$, and for such a surface $\Omega_X$ denotes the holomorphic  cotangent bundle, $T_X$ the holomorphic  tangent bundle, 
$\omega_X=\mO_X(K_X)$ the
canonical bundle, $p_g(X)=h^0(X, \, K_X)$ is the geometric genus,
$q(X)=h^1(X, \, K_X)$ is the
 irregularity and $\chi(\mO_X)=1-q(X)+p_g(X)$ is the
 holomorphic Euler-Poincar\'{e} characteristic. We also set $c_1:=c_1(T_X)=-c_1(\Omega_X)$ and $c_2:=c_2(T_X)=c_2(\Omega_X)$. The second Segre number of $X$ is the integer $c_1^2-c_2$.

For projective spaces and projective bundles we use the same conventions as in \cite[Chapter 6]{Laz04}. More specifically, if $V$ is a vector space, $\mathbb{P}(V)$ stands by the projective space of $1$-dimensional quotients of $V$; we denote by $\mathbb{G}(n, \, \mathbb{P}(V))$ the Grassmannian of $n$-dimensional subspaces of $\mathbb{P}(V)$, and by $\mathbb{G}( \mathbb{P}(V), \, n)$ the Grassmannian of $n$-dimensional quotients of $\mathbb{P}(V)$. Moreover, we write $\nu_n \colon \mathbb{P}(V) \to \mathbb{P}(S^nV)$ for the $n{\textrm{th}}$ Veronese embedding of $\mathbb{P}(V)$, and we call its image $\nu_n(\mathbb{P}(V))$  the $n{\mathrm{th}}$ Veronese variety of $\mathbb{P}(V)$. 

If $\mathscr{E}$ is a vector bundle on $X$ and $x \in \mathscr{E}$, we write $\mathscr{E}(x)$ for the fibre of $\mathscr{E}$ over $x$. Furthermore, we denote by $\pi \colon \mathbb{P}(\mathscr{E}) \to X$ the projective bundle of $1$-dimensional quotients of $\mathscr{E}$, so that $\pi_*\mathcal{O}_{\mathbb{P}(\mathscr{E})}(n)=S^n \mathscr{E}$ for all $n \geq 1$.


\section{Pluri-cotangent maps and Gauss maps} \label{sec:pluri}

Let $X$ be a compact, complex surface of general type. We say that its cotangent bundle $\Omega_X$ is \emph{strongly semi-ample} if the $n$th symmetric power $S^n \Omega_X$ is globally generated for some $n \geq 1$, namely, if the evaluation map 
\begin{equation} \label{eq:surjectivity-1}
H^0(X, \, S^n \Omega_X) \otimes \mathcal{O}_X \longrightarrow S^n \Omega_X
\end{equation}
is surjective. Note that this implies 
\begin{equation} \label{eq:sections-Sn}
h^0(X, \, S^n \Omega_X) \geq \operatorname{rank}\, S^n \Omega_X = n+1  
\end{equation}
and if equality holds then $S^n \Omega_X \simeq \mathcal{O}_X^{n+1}$. In particular, since we are assuming that $X$ is of general type, the strict inequality holds in \eqref{eq:sections-Sn}. Moreover, the strong semi-ampleness of $\Omega_X$ implies that $\mathcal{O}_{\mathbb{P}(\Omega_X)}(1)$ is semi-ample, see \cite[Section 3.1]{MU19}, and so $\Omega_X$ is nef.

\begin{lemma} \label{lem:ampleness of K_X}
If $\Omega_X$ is strongly semi-ample, then $X$ does not contain any smooth rational curve. In particular, $X$ is minimal and $K_X$ is ample. 
\end{lemma}
\begin{proof}
If $C$ is a smooth curve contained in $X$, then $\Omega_C$ is a quotient of the restricted bundle $\Omega_X|_C$. Since passing to the $n$th symmetric product preserves epimorphisms, it follows that $S^n\Omega_C$  is a quotient of $S^n \Omega_X|_C$. This implies that $\Omega_C$ is strongly semi-ample,  hence $g(C)\geq 1$.
\end{proof}

\begin{remark} \label{lemma: K ample Omega non ssa}
It is not hard to construct examples where $K_X$ is ample and $\Omega_X$ is not strongly semi-ample. For instance, take a smooth quintic surface $X \subset \mathbb{P}^3$ containing a line $L$. Since $L$ is a smooth rational curve (with $L^2=-3$) we see that $\Omega_X$ is not strongly semi-ample. On the other hand, by using adjunction formula we can check that there are no $(-1)$-curves or $(-2)$-curves on  $X$, so $X$ is a minimal model and $K_X$ is ample. 

\end{remark}

Lemma \ref{lem:ampleness of K_X} allows us to apply to our situation the next result, based on Bogomolov's work  \cite{Bog78}, see \cite[Corollary A.1]{Kob80} and \cite[p. 1341]{RouRous13}. 
\begin{proposition} \label{prop:semistability of Omega}
Let $X$ be a surface of general type with ample canonical class. Then for all $n \geq 1$ the vector bundle $S^n \Omega_X$ is semi-stable with respect to the polarization $K_X$, and moreover
\begin{equation} \label{eq:Bogomolov-vanishing-1}
H^0(X, \, S^n T_X \otimes \omega_X^k)=0 \quad \text{for} \, \,  n-2k >0. 
\end{equation}
\end{proposition}

Setting $k=1$ in \eqref{eq:Bogomolov-vanishing-1} and applying Serre duality, we get Bogomolov's vanishing 
\begin{equation} \label{eq:Bogomolov-vanishing-2}
H^2(X, \, S^n\Omega_X)=0 \quad \textrm{for} \,\, n \geq 3.
\end{equation}
\begin{corollary} \label{cor:h0-and-chi}
Let $X$ be a surface of general type with ample canonical class. Then for all $n \geq 3$ we have $h^0(X, \, S^n \Omega_X) \geq \chi(X, \,S^n \Omega_X)$.
\end{corollary}
\begin{proof}
Immediate consequence of \eqref{eq:Bogomolov-vanishing-2}.
\end{proof}


\medskip

\begin{remark} \label{rmk:trace-zero endomorphisms}
The extremal case $n=2, \, k=1$ in Proposition  \ref{prop:semistability of Omega} is characterized as follows, cf. \cite[Theorem B and Corollaries B.1 and B.2]{Kob80}. If $X$ is a minimal surface of general type with ample canonical bundle, then we have
\begin{equation}
H^0(X, \, S^2 T_X \otimes \omega_X)=H^0(X, \, S^2 \Omega_X \otimes \omega_X^{-1})=0
\end{equation}
if and only if $\Omega_X$ is an indecomposable rank $2$ vector bundle. One direction is clear: if $\Omega_X=L_1 \oplus L_2$ is the direct sum of two line bundles, then a straightforward computation shows that $S^2 \Omega_X \otimes \omega_X^{-1}$ has a direct summand isomorphic to $\mathcal{O}_X$, hence $H^0(X, \, S^2 \Omega_X \otimes \omega_X^{-1}) \neq 0$. Conversely, let us assume that $S^2 \Omega_X \otimes \omega_X^{-1}$ has a non-zero global section and let us show that $\Omega_X$ is decomposable; to this pourpose, we will use the following argument suggested to us by Igor Reider. Identifying $S^2 \Omega_X \otimes \omega_X^{-1}$ with the sheaf $End_0(\Omega_X)$ of trace-zero endomorphisms of $\Omega_X$, \footnote{This is a consequence of the following linear algebra facts. Consider a rank $2$ vector space $V$ over a field $\mathbb{K}$ of characteristic different from $2$.  Since every square matrix can be written in a unique way as the sum of a symmetric matrix and a skew-symmetric one, we have the direct sum decomposition 
$V \otimes  V = S^2 V \oplus \wedge^2V$.
Taking the tensor product with $\wedge^2V^*$, and using the identification
$V \otimes \wedge^2 V^* = V^*$ (coming from the bilinear pairing on $V$ induced by the wedge product, namely  $v \otimes  w  \mapsto v \wedge w$), we get an identification
$V^* \otimes V = (S^2V \otimes \wedge^2V^*) \oplus \mathbb{K}$.
On the other hand, $V^* \otimes V = \operatorname{Hom}(V,\, V)$, so we get a further identification
$\operatorname{Hom}(V, \, V) = (S^2V \otimes \wedge^2V^*) \oplus \mathbb{K}$. Under this identification, an endomorphism
$f \colon V \to  V$ satisfies $\operatorname{Trace}(f)=0$ if and only if it lies in the first summand $S^2V \otimes \wedge^2V^*$. This is a straightforward computation based on the interpretation of the trace as the functional $ V^* \otimes V \to \mathbb{K}$ given by the natural evaluation on decomposable tensors, namely $\operatorname{Trace}(f \otimes v)=f(v)$. 
Therefore $S^2V \otimes \wedge^2V^*$ is naturally identified with the vector space $\operatorname{Hom}_0(V, \, V)$ of trace-zero endomorphisms of $V$. \\}  a non-zero global section corresponds to an endomorphism $f \colon \Omega_X \to \Omega_X$ whose trace is zero at every point. Now we have two cases:
\begin{itemize}
\item[$\boldsymbol{(i)}$] There is a point $x_0 \in X$ such that $f_{x_0} \colon \Omega_X(x_0) \to \Omega_X(x_0)$ has two non-zero eigenvalues $\pm \lambda$; then 
\begin{equation}
\Omega_X= \ker(f- \lambda \operatorname{I}) \oplus \ker(f + \lambda \operatorname{I})
\end{equation}
is the desired splitting. \footnote{If we have an endomorphism $g \colon E \to E$ of a vector bundle $E$, then its determinant $\det g \colon \det E \to \det E$ is a scalar multiple of the identity and so, if it vanishes at one point, it vanishes everywhere. Thus, taking $E=\Omega_X$ and $g=f-\lambda I$, $g=f + \lambda I$, we get 
\begin{equation}
\Omega_X(x)=\ker(f- \lambda I)(x) \oplus \ker(f + \lambda I)(x)
\end{equation}
for \emph{all} $x \in X$.}

\item[$\boldsymbol{(ii)}$] The endomorphism $f \colon \Omega_X \to \Omega_X$ is nilpotent everywhere, hence $f^2=0$. We will rule out this case, by exploiting the ampleness of $K_X$. In fact, by the nilpotency condition, the sheaf $\operatorname{im}(f)$ injects into $\ker(f)$; then, setting
\begin{equation}
c_1(\ker(f))= L, \quad c_1(\operatorname{im}(f))=L' 
\end{equation}
the divisor $L-L'$ is effective. The semi-stability of $\Omega_X$ with respect to the polarization $K_X$ now gives
\begin{equation}
K_XL=\mu(L) \leq \mu(\Omega_X)=K_X^2/2, \quad K_XL'=\mu(L') \geq \mu(\Omega_X)=K_X^2/2,
\end{equation} 
and so
\begin{equation}
0 \leq K_X(L-L') \leq K_X^2/2 - K_X^2/2 =0,
\end{equation}
that implies $L=L'$. Thus $K_X=c_1(\Omega_X)=2L$, hence $L$ is ample. But this is impossible  because, by a result of Bogomolov, $\Omega_X$ cannot have ample subsheaves of rank $1$, see \cite[Theorem 2]{Reid77}. 
\end{itemize}  
\end{remark}
\bigskip 

Recalling that we have a natural identification between $H^0(X, \, S^n \Omega_X)$ and $H^0(\mathbb{P}(\Omega_X), \, \mathcal{O}_{\mathbb{P}(\Omega_X)}(n))$, from the surjectivity of \eqref{eq:surjectivity-1} we infer that the induced evaluation map
\begin{equation} \label{eq:surjectivity 2}
H^0(X, \, S^n \Omega_X) \otimes \mathcal{O}_{\mathbb{P}(\Omega_X)} \longrightarrow \mathcal{O}_{\mathbb{P}(\Omega_X)}(n)
\end{equation}
is also surjective, and so defines a morphism
\begin{equation} \label{eq:cotang}
\psi_n \colon \mathbb{P}(\Omega_X) \to \mathbb{P}(H^0(X, \, S^n \Omega_X)). 
\end{equation}

\begin{definition} \label{def:cotang}
We call $\psi_n$ the ${n}{\mathrm{th}}$ \emph{cotangent map} of $X$, and we denote its image by $X_n := \psi_n(\PP(\Omega_X)) \subset \PP (H^0(X, \, S^n \Omega_X))$. 
\end{definition}

\sloppy By \cite[Appendix A]{Laz04} there is a relative $n$th Veronese embedding  $\nu_n \colon \mathbb{P}(\Omega_X
) \to \mathbb{P}(S^n \Omega_X)$ such that, for every $x \in X$, the fibre $\mathbb{P}(\Omega_X(x))$ of $\mathbb{P}(\Omega_X)$ over $x$ is sent to a rational normal curve of degree $n$ inside the $n$-dimensional projective space $\mathbb{P}(S^n\Omega_X(x))$. Moreover, passing to projective bundles in the evaluation map \eqref{eq:surjectivity-1}, we get a morphism $e_n \colon \mathbb{P}(S^n \Omega_X) \to  \mathbb{P}(H^0(X, \, S^n \Omega_X))$ and a factorization of $\psi_n$ of the form 
\begin{equation} \label{eq:diagram_Veronese}
\begin{tikzcd}
    \mathbb{P}(\Omega_X) \arrow{rr}{\psi_n} \arrow[swap]{dr}{\nu_n} & & \mathbb{P}(H^0(X, \, S^n \Omega_X)) \\
    & \mathbb{P}(S^n \Omega_X)  \arrow[swap]{ur}{e_n}
\end{tikzcd}
\end{equation}

Let us now consider two important examples: the case where $\Omega_X$ is ample and the case where $\Omega_X$ is globally generated.

\begin{example} \textbf{The case where $\Omega_X$ is ample.} \label{ex:Omega ample}
If $\Omega_X$ is ample, then it is automatically strongly semi-ample, see \cite[Theorem 6.1.10]{Laz04}, and the Chern numbers of $X$ satisfy the inequality $c_1^2-c_2 >0$, see \cite{Kl69}. If $S^n \Omega_X$ is globally generated, by \cite[Example 6.1.5 and Theorem 6.1.15]{Laz04} the ampleness of $\Omega_X$ is equivalent to the fact that $e_n$ is finite onto its image.  Summing up, we can state what follows:

\medskip

\emph{Assume that $S^n \Omega_X$ is globally generated. Then the $n$th cotangent map $\psi_n$ is a finite morphism onto its image $X_n$ if and only if $\Omega_X$ is ample}.

\medskip

For the sake of completeness, let us shortly explain how to construct surfaces for which $\Omega_X$ is strongly semi-ample but not ample. Let $A$ be an abelian 3-fold containing an elliptic curve $E$, and let $X \subset A$ be a sufficiently positive, smooth divisor containing $E$. Then $X$ is a surface of general type  whose Albanese morphism $a_X \colon X \to \operatorname{Alb}(X)$ coincides with the inclusion $X \to A$. Since $A$ contains no rational curves, the same is true for $X$, which is therefore a minimal model with ample $K_X$. Furthermore, since $\Omega_X$ is a quotient of $\Omega_A|_X=\mathcal{O}_X^{\oplus 3}$, it follows that $\Omega_X$ is globally generated, hence strongly semi-ample. However, $\Omega_X$ is not ample: in fact, varieties with ample cotangent bundle are Kobayashi hyperbolic (\cite[Theorem 6.3.26]{Laz04}), in particular, they do not contain any elliptic curves. For a detailed analysis of a similar construction in codimension $2$, see Subsection \ref{subsec:complete-intersections}.

\end{example}

\begin{example} \textbf{The case where $\Omega_X$ itself is globally generated.} \label{ex:Omega_generated} 
For $n=1$ the evaluation map \eqref{eq:surjectivity-1} is the co-differential of the Albanese morphism
\begin{equation}
a_X \colon X \to \operatorname{Alb}(X),
\end{equation} 
hence the cotangent bundle $\Omega_X$ is globally generated if and only if $a_X$ is a local immersion. In this case, $S^n \Omega_X$ is globally generated for all $n \geq 1$, and we have a natural symmetrization homomorphism $\sigma_n \colon S^n H^0(X, \, \Omega_X) \longrightarrow H^0(X, \, S^n  \Omega_X)$, that fits into a commutative diagram  
\begin{equation} \label{dia:symmetrization}
\begin{split}
   \xymatrix{ 
       \PP(\Omega_X) \ar[d]_{\psi_1} \ar[r]^-{\psi_n} & \PP (H^0(X, \, S^n \Omega_X))\ar@{-->}[d]^{\mathbb{P}(\sigma_n)} \\  
  \PP(H^0(X, \, \Omega_X)) \ar[r]^-{\nu_n}  & \PP (S^nH^0(X, \,\Omega_X)).} & 
     \end{split}
\end{equation}
Here $\nu_n$ stands for $n{\mathrm{th}}$ Veronese embedding and the rational map $\mathbb{P}(\sigma_n)$ is an embedding of projective spaces if $\sigma_n$ is surjective, and a linear projection otherwise. \footnote{According to our understanding, not much is known about the behaviour of $\sigma_n$ in general. A result in this direction would provide a higher-dimensional generalization of the celebrated Max Noether's Theorem for curves, see \cite[p. 117]{ACGH85} and the MathOverflow thread [\href{https://mathoverflow.net/questions/273557/}{MO273557}].} By \cite[Proposition 2.14]{Rou09} it follows that if $q(X)=h^0(X, \, \Omega_X) >3$ then $\psi_1$ is generically finite, hence we can draw the following conclusion:

\medskip

\sloppy \emph{Assume that $\Omega_X$ is globally generated and $q(X) >3$. Then the $n$th cotangent map $\psi_n$ is generically finite onto its image for all $n \geq 1$}.   


\medskip

Again for the sake of completeness, let us provide examples where $\Omega_X$ is strongly semi-ample but not globally generated. If $X$ is a a fake projective plane (see \cite{PY07}), then $\Omega_X$ is ample (this is true for every smooth compact complex variety uniformized by the ball $\mathbb{B}^n \subset \mathbb{C}^n$, see \cite[Construction 6.3.36]{Laz04}) and thus strongly semi-ample. However, $h^0(X, \Omega_X)=h^1(X, \, \mathcal{O}_X)=0$, namely, $\Omega_X$ has no global sections at all.

\end{example} 

\sloppy \section{Finiteness of the Gauss map and dimension of the pluri-cotangent image } \label{sec:image pluri-cotangent} \label{sec:finiteness}
 
\begin{assumption} \label{ass:assumption}

From now on, $X$ will denote a surface of general type with strongly semi-ample cotangent bundle $\Omega_X$. Note that we are \emph{neither} assuming that $\Omega_X$ is ample \emph{nor} that $\Omega_X$ is globally generated, having already analyzed these cases before. 
\end{assumption} 
 
\begin{proposition} \label{prop:c1^2-c2>0}
Assume that $S^n \Omega_X$ is globally generated, with $n \geq 3$. Then $\psi_n$ is generically finite onto its image $X_n$ if and only if $c_1^2-c_2 >0$. In this case, we have $\deg X_n=n^3(c_1^2-c_2)/\deg \psi_n$.
\end{proposition}
\begin{proof}
Using the asymptotic form of Riemann-Roch theorem for vector bundles together with the vanishing \eqref{eq:Bogomolov-vanishing-2}, we get
\begin{equation} \label{eq:asymptotic-RR}
h^0(X, \, S^n \Omega_X) \geq \chi(X, \, S^n \Omega_X)= \frac{n^3}{6}(c_1^2-c_2)+O(n^2).
\end{equation}
Thus, the positivity of the second Segre number $c_1^2-c_2$ implies that $S^n \Omega_X$ is big, and so $\psi_n$ is generically finite onto its image since $S^n \Omega_X$ is globally generated. Conversely, suppose that $\psi_n$ is generically finite onto $X_n$ for some $n$. Then, if $\xi \in |\mathcal{O}_{\mathbb{P}(\Omega_X)}(1)|$, the same argument used in the proof of \cite[Proposition 2.15]{Rou09} shows that
\begin{equation}  \label{eq:product-degrees}
0 < \deg X_n \cdot \deg \psi_n= (n\xi)^3 = n^3(c_1^2-c_2).
\end{equation}
\end{proof}


\begin{remark} \label{s_2=0}
Subsection \ref{subsec:lambda-6-sharp} contains a detailed analysis of some examples where $\Omega_X$ is globally generated, $c_1^2-c_2=0$ and $\dim X_n =2$ for all $n \geq 1$. This shows that the assumption about the positivity of the second Segre number in Proposition \ref{prop:c1^2-c2>0} cannot be removed.
\end{remark}

\begin{remark} \label{rmk:Roulleau-Rousseau}
There exist examples of surfaces $X$ of general type with big cotangent bundle and $c_1^2-c_2 \leq 0$, see \cite{RouRous13}. They are obtained by taking the minimal resolution of some singular models with rational double points, hence they contain smooth rational curves and so $\Omega_X$ is not strongly semi-ample (Lemma \ref{lem:ampleness of K_X}).
\end{remark}

In this paper we focus on finding explicit lower bounds on $n$ such that $\psi_n$ is generically finite onto its image. Our arguments are geometric in nature, and generalize the ones used in \cite[Section 2]{Rou09}; furthermore, we use in an essential way some results from  \cite{PR13} and \cite{MU19}.
 
Let $\pi \colon \mathbb{P}(\Omega_X) \to X$ be the structure projection and let us look at the restriction of $\psi_n$ to the fibre $\pi^{-1}(x)$ over a point $x \in X$. Such a fibre is the curve $\PP(\Omega_X(x)) \simeq \PP^1$, and the restriction of $|\mathcal{O}_{\mathbb{P}(\Omega_X)}(n)|$ to it is the complete linear system $|\mathcal{O}_{\mathbb{P}^1}(n)|$, that embeds $\pi^{-1}(x)$ as a rational normal curve $C_x$ of degree $n$ in $\PP (H^0(X, \, S^n \Omega_X))$. There is a unique $n$-dimensional linear subspace $L_x \subset \PP (H^0(X, \, S^n \Omega_X))$ containing $C_x$, so we have a morphism
\begin{equation*}
\mathsf{g}_n \colon X \to \mathbb{G}(n, \, \PP (H^0(X, \, S^n \Omega_X))), \quad x \mapsto L_x.
\end{equation*}

\begin{definition} \label{def:gauss}
We call $\mathsf{g}_n$ the ${n}{\mathrm{th}}$ \emph{Gauss map} of $X$, and we denote its image by $Y_n := \mathsf{g}_n(X) \subset \mathbb{G}(n, \, \PP (H^0(X, \, S^n \Omega_X)))$. 
\end{definition}
Let us now provide an alternative description of the Gauss map. Being $\Omega_X$ globally generated, for every point $x \in X$ there is a surjection 
\begin{equation} \label{eq:surj-Omega}
H^0(X, \, S^n\Omega_X) \to S^n\Omega_X(x) \to 0.
\end{equation}
Since the fibre $S^n\Omega(x)$ of $S^n\Omega$ over $x$ is a vector space of dimension $n+1$, after passing to projective spaces and dualizing we obtain a quotient of dimension $n$ of  $\mathbb{P}(H^0(X, \, S^n\Omega_X))^*$, hence an element
\begin{equation}
 s_x \in \mathbb{G}(n, \, \mathbb{P}(H^0(X, \, S^n\Omega_X)))  \simeq \mathbb{G}(\mathbb{P}(H^0(X, \, S^n\Omega_X))^*, \, n).  
\end{equation}
Thus, we get a morphism 
\begin{equation}
\mathsf{k}_n \colon X \to \mathbb{G}(n, \, \mathbb{P}(H^0(X, \, S^n\Omega_X))), \quad x \mapsto s_x. 
\end{equation}
Following \cite[p. 2230]{MU19} we call $\mathsf{k}_n$ the $n$th \emph{Kodaira map} of $X$.


\begin{proposition}\label{prop:G=H}
The morphisms $\mathsf{g}_n$ and $\mathsf{k}_n$ do actually coincide.
\end{proposition}
\begin{proof}
Diagram \eqref{eq:diagram_Veronese} shows that the $n$-dimensional linear space $L_x$ containing $C_x$ coincides with the image of $\mathbb{P}(S^n \Omega_X(x))$
via the map $e_n$. By construction, this image is precisely $s_x$. 
\end{proof}

\begin{proposition} \label{prop:gauss-finite}
The $n{\mathrm{th}}$ Gauss map $\mathsf{g}_n$ is a finite morphism onto its image $Y_n$. In particular, we have  $\dim Y_n =2$ for all $n \geq 1$.
\end{proposition}
\begin{proof}
By Proposition \ref{prop:G=H}, it is equivalent to prove the result for the Kodaira map $\mathsf{k}_n$. 
Let $\widetilde{Y}_n \to Y_n$ be the normalization map of $Y_n$, and let $\tilde{\mathsf{k}}_n \colon X \to \widetilde{Y}_n$ be the corresponding lifting of $\mathsf{k}_n$ (which exists since $X$ is smooth, hence normal). If the result is true for $\tilde{\mathsf{k}}_n$ then it is true for $\mathsf{k}_n$ as well, because the normalization is a finite map. Thus, we may assume that $Y_n$ is normal.

Next, the assumption that $S^n \Omega_X$ is globally generated implies that $S^{tn} \Omega_X$ is globally generated for all positive integers $t$. According to \cite[Lemma 3.3]{MU19}, we have a factorization  
\begin{equation}
\begin{tikzcd}
    X \arrow{rr}{\mathsf{k}_n} \arrow[swap]{dr}{\mathsf{k}_{tn}} & & Y_n \\
    & Y_{tn}  \arrow[swap]{ur}{u}
\end{tikzcd}
\end{equation}
where $u \colon Y_{tn} \to Y_{n}$ is a finite map. Furthermore, by \cite[Theorem 3.4]{MU19}, there exists a diagram
\begin{equation}
\begin{tikzcd}
    X \arrow{rr}{\mathsf{k}_{\textrm{det}}} \arrow[swap]{dr}{\mathsf{k}_{\mathbb{G}}} & & Y_{\infty} \arrow{dl}{v} \\
    & Y_{\mathbb{G}}  
\end{tikzcd}
\end{equation}
where $\mathsf{k}_{\textrm{det}}$ is the Iitaka fibration induced by $K_X=\det \Omega_X$ and $v$ is a finite map, such that for $t \gg 0$ we have $Y_{tn}=Y_{\mathbb{G}}$ and $\mathsf{k}_{tn}= \mathsf{k}_{\mathbb{G}}$. So, for $t$ sufficiently large, we get 
\begin{equation}
\mathsf{k}_{n}= u \circ \mathsf{k}_{tn}=u \circ v \circ \mathsf{k}_{\textrm{det}}.
\end{equation}
Since $K_X$ is ample (Lemma \ref{lem:ampleness of K_X}), it follows that $\mathsf{k}_{\textrm{det}}$ is an isomorphism onto its image; therefore $\mathsf{k}_{n}$ is a composition of finite maps, and the proof is complete. 

\end{proof}

\begin{lemma} \label{lem:restriction-pi}
For all $p \in X_n$ (the image of $\psi_n$ in $\PP (H^0(X, \, S^n \Omega_X))$), the restriction of $\pi \colon \PP(\Omega_X) \to X$ to $\psi_n^{-1}(p)$ is injective. 
\end{lemma}
\begin{proof}
Given $x \in X$, the intersection $\pi^{-1}(x) \cap \psi_n^{-1}(p)$ is either empty or consists of a single point, because the restriction $\psi_n \colon \pi^{-1}(x) \to C_x \subset \PP (H^0(X, \, S^n \Omega_X))$ is an embedding.
\end{proof}

\begin{remark}\textbf{The degree of the Gauss map divides the degree of the pluri-cotangent map.} \label{rmk:pluricotangent-factors-through-Gauss}
Denoting by $U_n$ the universal vector bundle over the affine Grassmannian of $(n+1)$-dimensional subspaces of $H^0(X, \, S^n \Omega_X)$ and by $\PP(U_n)$ the corresponding projectivization, we have a commutative diagram
\begin{equation*}
\begin{split}
   \xymatrix{ 
       \PP(\Omega_X) \ar@/^2.0pc/@[black][rr]^{\psi_n}\ar[d]_{\pi} \ar[r]^-{\tilde{\psi}_n} & \PP(U_n) \ar[d]^-{\pi_{2}} \ar[r]^-{\pi_1}  & \mathbb{P}(H^0(X, \, S^n \Omega_X)) \\  
 X \ar[r]^-{\mathsf{g}_n}  & \mathbb{G}(n, \, \mathbb{P}(H^0(X, \, S^n \Omega_X))),} & 
     \end{split}
\end{equation*}
where $\pi_1$, $\pi_2$ are the natural projections and $\tilde{\psi}_n$ is such that $\pi_1 \circ \tilde{\psi}_n = \psi_n$.
Since the restriction of $\psi_n$ to the fibres of $\pi$ is an embedding, it follows that $\tilde{\psi}_n$ is a finite map onto its image, whose degree is the same as the degree of the Gauss map $\mathsf{g}_n$. So, assuming that  the pluri-cotangent map $\psi_n$ is generically finite onto its image $X_n$, and denoting by $\deg \psi_n$ the degree of $\psi_n \colon \PP(\Omega_X) \to X_n$, we obtain
\begin{equation} \label{eq:pluricotangent-factors-through-Gauss}
\deg \psi_n = \deg \mathsf{g}_n \cdot \deg \pi_1',
\end{equation}
where $\pi_1'$ is the restriction of $\pi_1$ to the image of $\tilde{\psi}_n$. As a consequence, if $\psi_n$ is generically finite onto its image, then $ \deg \mathsf{g}_n$ divides $\deg \psi_n$.

\end{remark}

If $p \in X_n$, let us denote by $D_p'$ the subvariety $\pi(\psi_n^{-1}(p)) \subset X$. By Lemma \ref{lem:restriction-pi}, we have
\begin{equation} \label{eq:dim-psi-1}
\dim D_p'=\dim \psi_n^{-1}(p).  
\end{equation}

\begin{proposition}  \label{prop:p-in-Cx}
Let $x \in X$. Then $x \in D_p'$ if and only if $p \in C_x$. 
\end{proposition}
\begin{proof}
If $x \in D_p'$ then $x=\pi(z)$, with $z \in \psi_n^{-1}(p)$. Thus $z \in \pi^{-1}(x)$ and so $p=\psi_n(z) \in \psi_n(\pi^{-1}(x))=C_x$. Conversely, if $p \in C_x$ then $p=\psi_n (w)$, with $w \in \pi^{-1}(x)$. Then $x=\pi(w) \in \pi(\psi_n^{-1}(p))=D_p'$. 
\end{proof}

\begin{proposition} \label{prop:dim-Dp}
For all $p \in X_n$, the subvariety $\psi^{-1}_n(p) \subset \mathbb{P}(\Omega_X)$ has dimension at most $1$. Hence $D_p' \subset X$ has dimension at most $1$, too.
\end{proposition}
\begin{proof}
Write $Z:=\psi_n^{-1}(p)$ and $N:=h^0(X, \, S^n \Omega_X)$. Since $\psi_n$ is not a constant map, we have $\dim Z \leq 2$. By contradiction, assume   $\dim Z = 2$. Then the restriction $\pi|_{Z} \colon Z \to X$ is a bijective morphism by Lemma \ref{lem:restriction-pi} and so, since $X$ is normal, it is an isomorphism by Zariski's Main Theorem, see \cite[Chapter III]{Mum99}. Thus we get a regular section $t \colon X \to \PP(\Omega_X)$ of the projective bundle $\pi \colon \PP(\Omega_X) \to X$, that in turn corresponds to a rank $1$  quotient $\Omega_X \to \mathcal{L}$, where $\mathcal{L} = t^* \mathcal{O}_{\mathbb{P}(\Omega_X)}(1)$, see \cite[Chapter II, Proposition 7.12]{Ha77}. Since $\psi_n \circ t$ contracts $X$ to the point $p \in \PP(H^0(X, \, S^n \Omega_X)) \simeq \mathbb{P}^{N-1}$, it follows that $(\psi_n \circ t)^* \mathcal{O}_{\mathbb{P}^{N-1}}(1)$ is the trivial line bundle on $X$, and so
\begin{equation} \label{eq:trivial_line_bundle}
\mathcal{L}^n = t^* \mathcal{O}_{\mathbb{P}(\Omega_X)}(n) = t^* \psi_n^* \mathcal{O}_{\PP^{N-1}}(1)= \mathcal{O}_X.
\end{equation}
So, taking the $n$th symmetric product of $\Omega_X \to \mathcal{L}$, we get a quotient $S^n \Omega_X \to \mathcal{O}_X$, contradicting the semi-stability of $S^n \Omega_X$ with respect to $K_X$, see Proposition \ref{prop:semistability of Omega}. Hence the only possibility is  $\dim Z \leq 1$, and this proves the first statement. The second statement follows from \eqref{eq:dim-psi-1}.
\end{proof}


\begin{definition} \label{def:exceptional-point}
A point $p \in X_n$ is called \emph{exceptional for} $\psi_n$ if the fibre $\psi_n^{-1}(p)$ has dimension $1$. The set of such exceptional points will be denoted by $\mathrm{exc}(\psi_n)$.
\end{definition}

Let $p$ be an exceptional point for $\psi_n$, so that $\dim D_p'=1$, and let $D_p \subseteq D_p'$ be an irreducible component of dimension $1$.  If we set $\Sigma:=\psi_n(\pi^{-1}(D_p))$, then we have
\begin{equation}\label{eq:Sigma}
\Sigma = \bigcup_{x \in D_p}C_x.
\end{equation}

\begin{proposition} \label{prop:family-rational-normal}
The variety $\Sigma$ is a surface in $\PP (H^0(X, \, S^n \Omega_X))$ containing a $1$-dimensional family of rational normal curves of degree $n$ passing through $p$. More precisely, for every point $q \in \Sigma$, different from $p$, there exists a rational normal curve of degree $n$ contained in $\Sigma$ and joining $p$ and $q$. 
\end{proposition}
\begin{proof}
The fact that $\Sigma$ has dimension at least $1$ is an immediate consequence of \eqref{eq:Sigma}. By contradiction, assume $\dim \Sigma =1$; then there is a rational normal curve $C$ of degree $n$ such that $C_x=C$ for all $x \in D_p$. This in turn implies that the $n$-plane $L_x \subset\PP (H^0(X, \, S^n \Omega_X))$ is constant on $D_p$, and so the $n{\mathrm{th}}$ Gauss map $\mathsf{g}_n \colon X \to \mathbb{G}(n, \, \PP (H^0(X, \, S^n \Omega_X))$ contracts $D_p$ to a point, against Proposition \ref{prop:gauss-finite}. It follows that $\Sigma \subset \PP (H^0(X, \, S^n \Omega_X))$ is a surface. If $q \in \Sigma$, then $q \in C_x$ for some $x \in D_p$; thus we have $p \in C_x$ by Proposition \ref{prop:p-in-Cx}, hence the rational normal curve $C_x$ joins $p$ and $q$.  
\end{proof}

\begin{definition}
Let $N > 3$ be a positive integer. An irreducible variety $V \subset \PP^{N-1}$ is said to be $2$-\emph{covered by curves of degree} $n$ if a general pair of points of $V$ can be joined by an irreducible curve of degree $n$.  
\end{definition}

The case of surfaces $2$-covered by curves was classically considered in \cite{Bom21}; a modern treatment can be found in \cite[pp. 718-722]{PR13}, see also \cite[Theorem 2.8]{Ion05} and \cite[Th\'{e}or\`{e}me 1.5]{PT13}. It turns out that those maximizing the dimension of the ambient space $\mathbb{P}^{N-1}$ are precisely the Veronese embeddings of $\mathbb{P}^2$. \footnote{The simplest way to obtain a  surface which is $2$-covered by curves of degree $n$ is taking a surface $X$ of degree $n$ embedded in $\mathbb{P}^{N-1}$, so that the $2$-covering family is provided by the hyperplane sections.  If $X$ is non-degenerate then $N \leq  n+2$, which is much smaller of the optimal value $\frac{1}{2}(n+1)(n+2)$, achieved by the $n$th Veronese surface. }

\begin{proposition} \label{prop:n-covered}
If $\Sigma \subset \mathbb{P}^{N-1}$ is  a non-degenerate surface which is $2$-covered by curves of degree $n$, then $N \leq \frac{1}{2}(n+1)(n+2)$. Moreover,  equality holds if and only if $\Sigma$ is projectively equivalent to the $n{\mathrm{th}}$ Veronese surface $\nu_n(\mathbb{P}^2)$ and, in this case, every curve in the $2$-covering family is a rational normal curve of degree $n$ and there exists a unique such a curve passing through two distinct points of $\Sigma$.
\end{proposition}
\begin{proof}
See \cite[Theorem 2.2]{PR13}. 
\end{proof}

Let us now study the image of the $n{\mathrm{th}}$ cotangent map and the geometry the locus $\mathrm{exc}(\psi_n)$; this generalizes the analysis of the case $n=1$ that was carried out in \cite[Proposition 2.14]{Rou09}.

\begin{proposition} \label{prop:image-n-cotangent}
If $n \geq 2$ and the map $\psi_n \colon \PP(\Omega_X) \to \PP (H^0(X, \, S^n \Omega_X))$ is not generically finite, then its image $X_n \subset  \PP (H^0(X, \, S^n \Omega_X))$ is a non-degenerate, linearly normal surface $2$-covered by rational normal curves of degree $n$. Moreover, under the same  assumptions one has $h^0(X, \, S^n \Omega_X) \leq  \frac{1}{2}(n+1)(n+2)$, with equality if and only if  $X_n$ is projectively equivalent to $\nu_n(\mathbb{P}^2)$. \footnote{Some examples where $X_n$ is projectively equivalent to $\nu_n(\mathbb{P}^2)$ for all $n$ will be given in Subsection \ref{subsec:lambda-6-sharp}.}
\end{proposition}
\begin{proof}
By Proposition \ref{prop:dim-Dp} we have $2 \leq \dim X_n \leq 3$. If $\dim X_n=2$ then every fibre of $\psi_n \colon  \PP(\Omega_X) \to X_n$ has dimension $1$, i.e. every point $p \in X_n$ is an exceptional point for $\psi_n$. Hence $D_p$ has dimension $1$, whereas $\Sigma = \psi_n(\pi^{-1}(D_p))$ has dimension $2$ and is contained in the irreducible surface $X_n$. Therefore $\Sigma = X_n$ and so, by using Proposition \ref{prop:family-rational-normal} and the fact that $p \in X_n$ is arbitrary, we infer that $X_n$ is  $2$-covered by curves of degree $n$. The surface $X_n$ is non-degenerate and linearly normal, being the image of the morphism induced by a complete linear system in $\mathbb{P}(\Omega_X)$. The last statement now follows from Proposition \ref{prop:n-covered}.
\end{proof}

We can now state the main result of this section.

\begin{theorem} \label{thm:image-n-cotangent}Let $n \geq 2$ be an integer such that $S^n \Omega_X$ is globally generated. 
If $h^0(X, \, S^n \Omega_X) > \frac{1}{2}(n+1)(n+2)$ then $\psi_n$ is generically finite onto its image, namely, $\dim X_n = 3$, and we have
\begin{equation} \label{eq:deg-pluricotangent-map}
\deg \psi_n \leq \frac{n^3(c_1^2-c_2)}{h^0(X, \, S^n \Omega_X)-3}.
\end{equation}
In this case, $\mathrm{exc}(\psi_n)$ is a Zariski-closed, possibly empty subset of $ \PP (H^0(X, \, S^n \Omega_X))$ of dimension at most $1$. 
\end{theorem}
\begin{proof}
The first statement is an immediate consequence of Proposition \ref{prop:image-n-cotangent}. Since $X_n$ is a non-degenerate threefold in $\PP (H^0(X, \, S^n \Omega_X))$, we have $\deg X_n \geq h^0(X, \, S^n \Omega_X)-3$ and so \eqref{eq:deg-pluricotangent-map} is a consequence of \eqref{eq:product-degrees}. Regarding the last statement, if $\psi_n$ is generically finite onto its image then we have $\dim \psi_n^{-1}(\mathrm{exc}(\psi_n)) \leq 2$, hence the exceptional locus $\mathrm{exc}(\psi_n)$ has dimension at most $1$. Such a locus is a (possibly empty) Zariski-closed subset of $\PP (H^0(X, \, S^n \Omega_X))$ because $\psi_n$ is a proper morphism, see \cite[Corollaire 13.1.4]{EGAIV}. 
\end{proof}

Note that, as explained in Example \ref{ex:Omega ample}, the exceptional locus $\mathrm{exc}(\psi_n)$ is empty if and only if $\Omega_X$ is ample.

\section{A formula for $\chi(X, S^n \Omega_X)$, with an application to pluri-cotangent maps} \label{sec:generically-finiteness}

This section is based on a  straightforward calculation, whose details are included because we could not find a suitable reference.
\begin{lemma} \label{lem:chi-S^n}
Let $X$ be a compact, complex surface. We have
\begin{equation} \label{eq:chi-S^n}
\chi(X, \, S^n \Omega_X) = \frac{1}{12}(n+1) \left( (2n^2-2n+1)c_1^2-(2n^2+4n-1)c_2 \right).
\end{equation} 
\end{lemma}
\begin{proof}
This is a standard application of the splitting principle, as stated in \cite[p. 28]{Fried98}: every universal formula on Chern classes which holds for direct sum of line bundles holds in general. Let $\mathscr{E}=L_1 \oplus L_2$ be a decomposable rank $2$ vector bundle on $X$; then 
\begin{equation} \label{eq:c1V-c2V}
c_1(\mathscr{E})=c_1(L_1)+c_1(L_2), \quad c_2(\mathscr{E})=c_1(L_1)c_1(L_2).  
\end{equation}
We can compute $c_1(S^n\mathscr{E})$ as follows:
\begin{equation} \label{eq:c1-SnV}
\begin{split}
c_1(S^n\mathscr{E}) & = c_1 \left(S^n(L_1 \oplus L_2) \right) =c_1 \left( \bigoplus_{i=0}^n L_1^i \otimes L_2 ^{n-i} \right) \\ 
& = \sum_{i=0}^n \left( ic_1(L_1) + (n-i)c_1(L_2) \right) = \frac{n(n+1)}{2} \left(c_1(L_1) + c_1(L_2) \right) \\
& = \frac{n(n+1)}{2} c_1(\mathscr{E}). 
\end{split}
\end{equation}
Let us now compute $c_2(S^n\mathscr{E})$. We have
\begin{equation} \label{eq:c2-SnV-1}
\begin{split}
c_2(S^n\mathscr{E}) & = c_2 \left(S^n(L_1 \oplus L_2) \right) = c_2 \left( \bigoplus_{i=0}^n L_1^i \otimes L_2 ^{n-i} \right) \\ 
& = \sum_{0 \leq i < j \leq n} c_1(L_1^i \otimes L_2^{n-i}) c_1(L_1^j \otimes L_2^{n-j}) \\
& = \sum_{0 \leq i < j \leq n} \left( ic_1(L_1) + (n-i)c_1(L_2)  \right) \left( jc_1(L_1) + (n-j)c_1(L_2)  \right) \\
& = A c_1(L_1)^2 + B c_1(L_1)c_1(L_2) + C c_1(L_2)^2,
\end{split}
\end{equation}
where
\begin{equation} \label{eq:c2-SnV-ABC}
\begin{split}
A& =\sum_{0 \leq i < j \leq n} ij, \\ B& =\sum_{0 \leq i < j \leq n} \left( i(n-j)+(n-i)j \right), \\ C&=\sum_{0 \leq i < j \leq n} (n-i)(n-j).
\end{split}
\end{equation}
The quantities in \eqref{eq:c2-SnV-ABC} can be calculated by means of the standard formulas for the sum of integers and squares, obtaining
\begin{equation} \label{eq:c2-SnV-ABC-computed}
\begin{split}
A=C=& \frac{1}{24}(n-1)n(n+1)(3n+2) \\
B& = \frac{1}{12} n(n+1)(3n^2+n+2).
\end{split}
\end{equation}
Plugging \eqref{eq:c2-SnV-ABC-computed} into \eqref{eq:c2-SnV-1}, and taking into account \eqref{eq:c1V-c2V}, we get
\begin{equation} \label{eq:c2-SnV-2}
c_2(S^n\mathscr{E})=\frac{1}{24}(n-1)n(n+1)(3n+2)c_1(\mathscr{E})^2+\frac{1}{6}n(n+1)(n+2)c_2(\mathscr{E}).
\end{equation}
Now, the Riemann-Roch theorem for vector bundles on surfaces, see \cite[p. 31]{Fried98}, implies 
\begin{equation} \label{eq:chi-SnV}
\chi(X, \, S^n\mathscr{E})=\frac{c_1(S^n\mathscr{E}) (c_1(S^n\mathscr{E})-K_X)}{2}-c_2(S^n\mathscr{E})+(n+1)\frac{c_1^2 + c_2}{12}.
\end{equation}
Setting $\mathscr{E}=\Omega_X$ in \eqref{eq:chi-SnV} and using \eqref{eq:c1-SnV} and \eqref{eq:c2-SnV-2}, by standard computations we obtain \eqref{eq:chi-S^n}.
\end{proof}

\begin{corollary} \label{cor:h^0(S^n)}
Let $X$ be surface of general type with ample canonical class. Then for all $n \geq 3$ we have
\begin{equation} \label{eq:h^0(S^n)}
h^0(X, \, S^n \Omega_X) \geq  \frac{1}{12}(n+1) \left( (2n^2-2n+1)c_1^2-(2n^2+4n-1)c_2 \right).
\end{equation}
\end{corollary}
\begin{proof}
Combine Lemma \ref{lem:chi-S^n} with Corollary \ref{cor:h0-and-chi}.
\end{proof}

As a consequence of the previous calculations we can state the following apparently simple result, which, however, we have not been able to find  in the literature. 
\begin{theorem} \label{thm:generically-finite-1}
Let $X$ be a compact, complex surface and let $n \geq 3$  such that $S^n \Omega_X$ is globally generated and  $\chi(X, \,S^n \Omega_X)  \geq 0$. Then the pluri-cotangent map  $\psi_n \colon \mathbb{P}(\Omega_X) \to \mathbb{P}(H^0(X, \, S^n\Omega_X))$ is generically finite onto its image. 
\end{theorem}
\begin{proof}
If $\chi(X, \,S^n \Omega_X)  \geq 0$, by using formula \eqref{eq:chi-S^n} we obtain
\begin{equation} \label{eq:chi>0}
c_1^2 \geq \frac{2n^2+4n-1}{2n^2-2n+1} c_2 = \left( 1+ \frac{6n-2}{2n^2-2n+1}\right) c_2 > c_2,
\end{equation}
that is $c_1^2-c_2 >0$. Now apply  Proposition \ref{prop:c1^2-c2>0}.
\end{proof}

\section{Examples and counterexamples} \label{sec:examples}

\subsection{Example: symmetric complete intersections in abelian fourfolds of product type} \label{subsec:complete-intersections}

Let us consider an abelian fourfold of the form $A \times E$, where $A$ is an abelian threefold and $E$ is an elliptic curve. Let  $M$ be a polarization on $A \times E$ and let $x \in A$ be a point which is not $2$-torsion. Up to replacing the polarization $M$ on $A \times E$ with a suitable positive multiple, by using parameter counting and  Bertini-type arguments we can find two smooth hypersurfaces $Y_1, \, Y_2 \in |M|$ such that 
\begin{itemize}
\item $Y_1$ and $Y_2$ are both symmetric, i.e. invariant with respect to the involution $-1 \colon A \times E \to A \times E$;
\item $Y_1$ and $Y_2$ both contain the elliptic curve $\{x\} \times E$;
\item $Y_1$ and $Y_2$ do not contain any $2$-torsion points of $A \times E$; 
\item the intersection $Y=Y_1 \cap Y_2$ is smooth.
\end{itemize}
The conditions above imply that $Y$ is a smooth, minimal surface on which $-1$ acts freely; then the quotient $f \colon Y \to X$ provide a smooth surface $X$, containing an elliptic curve $E'$ isomorphic to $E$.  

\begin{proposition} \label{prop:abelian-fourfold}
$X$ is a minimal surface of general type with 
\begin{equation} \label{eq:c_1-c2(X)-involution}
c_1(X)^2=2M^4, \quad c_2(X)=\frac{3}{2}M^4,
\end{equation}
hence $c_1(X)^2-c_2(X)=\frac{1}{2}M^4>0$. Moreover $H^0(X, \, \Omega_X)=0$, in particular $\Omega_X$ is not globally generated. Finally, $S^n\Omega_X$ is globally generated for all even $n$, hence $\Omega_X$ is strongly semi-ample. However, $\Omega_X$ is not ample.
\end{proposition}
\begin{proof}
Using the short exact sequence of tangent bundles
\begin{equation}
0 \to T_Y \to T_{A \times E}|_Y \to N_Y \to 0,
\end{equation}
and recalling that $ T_{A \times E}=\mathcal{O}_{A \times E}^{\oplus 4}$ and $N_Y=\mathcal{O}_Y(M)^{\oplus 2}$, we get the equality of total Chern classes
\begin{equation} \label{eq:total-Chern-class}
c(T_Y) = c(\mathcal{O}_Y(M))^{-2} 
\end{equation}
that in turn yields $c_1(Y)=\mathcal{O}_Y(-2M)$ and $c_2(Y)= 3M^4$. Since $f \colon Y \to X$ is an \'{e}tale double cover, we deduce \eqref{eq:c_1-c2(X)-involution}. Moreover, we have $\Omega_Y = f^*\Omega_X$, hence the vector space $H^0(X, \, \Omega_X)$ is isomorphic to the involution-invariant subspace of $H^0(Y, \, \Omega_Y)$. Now, by Lefschetz theorem for Hodge groups, see \cite[Example 3.1.24]{Laz04}, the global holomorphic $1$-forms on $Y$ are precisely the restrictions of those on $A \times E$, so none of them is invariant and we get $H^0(X, \, \Omega_X)=0$. On the other hand, if $n$ is even then all global sections of $S^n \Omega_{A \times E}$ are invariant and thus, when restricted to $Y$, they descend to $X$. These sections generate $S^n \Omega_{A \times E}$, hence they also generate $S^n \Omega_Y$ and $S^n \Omega_X$. Finally, $\Omega_X$ is not ample because $X$ contains the elliptic curve $E'$, see the discussion at the end of Example \ref{ex:Omega ample}. 
\end{proof}

\begin{proposition} \label{prop:cotangent_maps_example}
 Let $X$ be a surface as above. Then the pluri-cotangent map $\psi_n \colon \mathbb{P}(\Omega_X) \to \mathbb{P}(H^0(X, \, S^n\Omega_X))$ is generically finite onto its image for all even $n \geq 2$. 
\end{proposition}
\begin{proof}
From $\Omega_Y = f^*\Omega_X$ we infer $S^n \Omega_Y=f^*(S^n \Omega_X)$ for all $n\geq 1$. This implies that $H^0(X, S^n \Omega_X)$ is isomorphic to the involution-invariant subspace $H^0(Y, \, S^n \Omega_Y)^+ \subseteq H^0(Y, \, S^n\Omega_Y)$. On the other hand, \cite[Proposition 13]{Deb05} implies that the restriction map
\begin{equation}
H^0(A \times E, \, S^n \Omega_{A \times E}) \to H^0(Y, \, S^n \Omega_Y)
\end{equation}
is injective for all $n$; moreover, if $n$ is even then every global section of $S^n \Omega_{A \times E}$ is invariant, and so its restriction belongs to $H^0(Y, \, S^n \Omega_Y)^+$. Summing up, for all even $n \geq 2$ we have 
\begin{equation}
\begin{split}
h^0(X, \, S^n \Omega_X)& = \dim H^0(Y, \, S^n \Omega_Y)^+ \geq \dim H^0(A \times E, \, S^n \Omega_{A \times E}) \\ 
& = \frac{1}{6}(n+1)(n+2)(n+3) > \frac{1}{2}(n+1)(n+2).
\end{split}
\end{equation}
The claim now follows from Theorem \ref{thm:image-n-cotangent}.
\end{proof}

\subsection{Example: some product-quotient surfaces} \label{subsec:product-quotient}

We start by considering a hyperelliptic curve $C$ of genus $3$, endowed with an additional action of the cyclic group  $G=\mathbb{Z}_2$ as follows. The curve $C$ has affine equation of the form
\begin{equation}
y^2=a_8x^8+a_6x^6+a_4x^4+a_2x^2+a_0,
\end{equation}  
where the coefficients $a_i$ are such that the zeros of the polynomial at the right side are distinct. If $g$ is the generator of $G$, we define the action of $G$ on $C$ as 
\begin{equation}
g(x, \, y)=(-x, \, y). 
\end{equation}
One checks that $g$ has the four fixed points 
\begin{equation}
(0, \, \sqrt{a_0}), \quad (0, \, -\sqrt{a_0}), \quad (\infty, \, \sqrt{a_8}), \quad (\infty, \, -\sqrt{a_8}),
\end{equation}
hence the quotient map $C \to C/G$ is branched at four points and, by the Hurwitz formula, the curve $C/G$ has genus $1$. 


\begin{lemma} \label{lem:G-action on omega_C}
Let $H^0(C, \, \omega_C) =V^+ \oplus V^-$ be the decomposition of $H^0(C, \, \omega_C)$ into the $G$-invariant subspace $V^+$ and the $G$-antiinvariant subspace $V^-$. Then there exists a basis $\{\xi_1, \, \xi_2, \, \xi_3 \}$ of $H^0(C, \, \omega_C)$ such that
\begin{equation}
V^+=\langle \xi_1 \rangle, \quad V^-=\langle \xi_2, \, \xi_3 \rangle
\end{equation}
and the three canonical divisors $\operatorname{div}(\xi_1)$, $\operatorname{div}(\xi_2)$, $\operatorname{div}(\xi_3)$ have pairwise disjoint supports.
\end{lemma}
\begin{proof}
The vector space $H^0(C, \, \omega_C)$ is generated by the holomorphic $1$-forms which, in affine coordinates, can be written as 
\begin{equation}
\omega_1 := \frac{dx}{y}, \quad \omega_2:=x \frac{dx}{y}, \quad \omega_3:=x^2 \frac{dx}{y}. 
\end{equation}
Note that $\omega_2$ is $G$-invariant, whereas $\omega_1$ and $\omega_3$ are $G$-antiinvariant. Now set $\xi_1:= \omega_2$ and take as $\xi_2$ and $\xi_3$ two general elements in the $G$-antiinvariant subspace $\langle\omega_1, \, \omega_3 \rangle$. 
\end{proof}
Let us now consider another curve $F$ with a $G$-action.
\begin{lemma} \label{lem:G-action on omega_F}
Let $g = 2k+1 \geq 5$ be an odd integer. Then there exists a curve $F$ of genus $g$, endowed with a free $G$-action having the following property. Denoting by $H^0(F, \, \omega_F)=W^+ \oplus W^-$ the decomposition of $H^0(F, \, \omega_F)$ into $G$-invariant and $G$-antiinvariant subspace, we can find a basis $\{\tau_1, \ldots , \tau_g \}$ of $H^0(F, \, \omega_F)$ such that 
\begin{equation}
W^+= \langle \tau_1, \ldots, \tau_{k+1} \rangle, \quad W^-= \langle \tau_{k+2}, \ldots, \tau_{g} \rangle
\end{equation}
and the $g$ canonical divisors $\operatorname{div}(\tau_1), \ldots, \operatorname{div}(\tau_g)$ have pairwise disjoint supports.
\end{lemma}
\begin{proof}
Let $D$ be a curve of genus $k+1$ and let $\mathscr{L}$ be a non-trivial line bundle on $D$ such that $\mathscr{L}^2=\mathcal{O}_D$. Then there exists an \'{e}tale double cover $f \colon F \to D$, with $F$ of genus $2k+1$ and $f_* \mathcal{O}_F=\mathcal{O}_D \oplus \mathscr{L}^{-1}$; the curve $F$ comes with a free $G$-action, corresponding to the automorphism exchanging the two sheets of the cover. Furthermore, since $f_* \omega_F = \omega_D \oplus (\omega_D \otimes \mathscr{L}),$ we deduce
\begin{equation}
W^+=f^*H^0(D, \, \omega_D), \quad W^-=f^*H^0(D, \, \omega_D \otimes \mathscr{L}).
\end{equation}
Therefore the desired result follows if both $\omega_D$ and $\omega_D \otimes \mathscr{L}$ are globally generated. In fact, for a base-point free line bundle, a general section avoids any given finite set of points, so we can choose recursively a basis where each section avoids the base loci of the previous ones; moreover, this property is preserved by \'{e}tale pullbacks. It is well known that $\omega_D$ is base-point free, see for instance \cite[Lemma 5.1 p. 341]{Ha77}. Regarding $\omega_D \otimes \mathscr{L}$, a point $p$ is in its base locus if and only if
\begin{equation}
H^0(D, \, \omega_D \otimes \mathscr{L} (-p))=H^0(D, \, \omega_D \otimes \mathscr{L}),
\end{equation}
namely, if and only if $h^1(D,\, \omega_D \otimes \mathscr{L} (-p))=1$. By Serre duality, this is equivalent to the fact that there exists $q \in D$ such that the divisor class of $\mathscr{L}$ is of the form $q-p$. But then $\mathcal{O}_D(2q-2p)=\mathcal{O}_D$, hence the linear system spanned by $2p$ and $2q$ is a $g^1_2$ on $D$ and so $D$ is hyperelliptic. Summing up, $\omega_D \otimes \mathscr{L}$ is globally generated if and only if one of these conditions (both implying $g(D) \geq 3$, and so $g=g(F) \geq 5$) hold:
\begin{itemize}
\item $D$ is non-hyperelliptic;
\item $D$ is hyperelliptic and the divisor class of $\mathscr{L}$ is not the difference of two Weierstrass points.
\end{itemize}
This concludes the proof.
\end{proof}

Taking $C$ and $F$ as above, for all $k \geq 2$ we can now define a product-quotient surface $X=(C \times F)/G$, where $G$ acts diagonally on the product. Such action is free (because the action of $G$ on $F$ is so), hence $X$ is a smooth, minimal surface of general type, whose invariants are 
\begin{equation}
p_g(X)=3k+1, \quad q(X)=k+2,  \quad K_{X}^2=16 k.
\end{equation}
Thus $h^0(X, \, \Omega_X)=k+2$ and $c_1(X)^2-c_2(X)=8k >0$. The natural projections of $C\times F$ induce two isotrivial fibrations $X \to F/G$ and $X \to C/G$, whose general fibres are isomorphic to $C$ and $F$, respectively. In particular, the cotangent bundle $\Omega_X$ is not ample, see \cite[Corollaire 3.8]{Rou09}. 

Let us now show that $\Omega_X$ is not globally generated either, but that it is nevertheless strongly semi-ample.

\begin{proposition} \label{prop: S^2 for product-quotient}
Let $C$ and $F$ be curves with a $G$-action as above, and $X=(C \times F)/G$. Then $\Omega_{X}$ is not globally generated, whereas its second symmetric power $S^2 \Omega_{X}$ is globally generated $($and so $S^n \Omega_{X}$ is globally generated for all even $n)$.
\end{proposition}
\begin{proof}
Denoting by $\pi_C \colon C \times F \to C$, $\pi_F \colon C \times F \to F$ the two natural projections, we have 
\begin{equation}
\Omega_{C \times F} = L \oplus M,
\end{equation}
 where $L=\pi_C^* \omega_C$, $M=\pi_F^* \omega_F$. Moreover, the covering map $C \times F \to X$ being \'etale, for all $n \geq 1$
we have $H^0(X, \, S^n \Omega_X)=H^0(C \times F, \, S^n\Omega_{C \times F})^G$. Thus, since the action of $G$ on $C \times F$ does not exchange the two factors, in order to show that $S^n \Omega_X$ is globally generated we must show that it is possible to generate every summand $L^k \otimes M^{n-k}$ of $S^n\Omega_{C \times F}$ by using $G$-invariant global sections. 

In the case $n=1$, the space of invariant global sections of $\Omega_{C \times F}$ is 
\begin{equation}
V^+ \oplus W^+ = \langle \xi_1 \rangle \oplus \langle \tau_1, \ldots, \tau_{k+1} \rangle.
\end{equation}
This shows that $\Omega_X$ is not globally generated, since $V^+$ is $1$-dimensional and so it is not possible to generate the summand $L$ by means of $G$-invariant global sections.

In the case $n=2$, we have 
\begin{equation}
S^2 \Omega_{C \times F} = L^2 \oplus (L \otimes M) \oplus M^2.
\end{equation} 
We recall that $\operatorname{div}(\xi_i)$ are pairwise disjoint divisors, and the same is true for $\operatorname{div}(\tau_j)$, see Lemmas \ref{lem:G-action on omega_C} and \ref{lem:G-action on omega_F}; thus, using the notation $\alpha \boxtimes \beta$ for $\pi_C^* \alpha  \otimes \pi_F^* \beta$, we can say that
\begin{itemize}
\item the $G$-invariant sections $(\xi_1)^2, \, \,  (\xi_2)^2$ generate $L^2$;
\item the $G$-invariant sections $\xi_1 \boxtimes \tau_1, \, \, \xi_2 \boxtimes \tau_{k+2}, \, \,  \xi_3 \boxtimes \tau_{k+3}$ generate $L \otimes M$;
\item the $G$-invariant sections $(\tau_1)^2, \, \, (\tau_2)^2$ generate $M^2$.
\end{itemize}
This shows that $S^2 \Omega_{X}$ is globally generated.
\end{proof}
 
\begin{proposition} \label{prop:generically-finite-for-all-n}
Let $k \geq 2$ and $X$ be the surface constructed above. Then the pluri-cotangent map $\psi_n \colon \mathbb{P}(\Omega_{X}) \to H^0(X, \, S^n\Omega_{X})$ is generically finite onto its image for all even $n$. 
\end{proposition} 
\begin{proof}
We first analyze the case $n=2$. Using the $G$-invariant global sections of $S^2 \Omega_X$ produced of  at the end of the proof of Proposition \ref{prop: S^2 for product-quotient}, we get
\begin{equation}
\begin{split}
h^0(X, \, S^2 \Omega_X)& = \dim \left[ H^0(X,\, L^2)^G \oplus H^0(X, \, L \otimes M)^G \oplus H^0(X, \, M^2)^G \right] \\
& \geq 2+3+2=7.
\end{split}
\end{equation} 
We obtained $h^0(X, \, S^2 \Omega_X) >6$  for all  $k \geq 2$, hence $\psi_2$ is generically finite onto its image by Theorem \ref{thm:image-n-cotangent}. 

Assume now $n \geq 4$. Since  $c_1(X)^2-c_2(X)=16k-8k>0$,  the result follows from Proposition  \ref{prop:c1^2-c2>0}.
\end{proof}

\subsection{Counterexamples to the generic finiteness of $\psi_n$: smooth ample divisors in abelian threefolds} \label{subsec:lambda-6-sharp}

We will now show that the assumption $h^0(X, \, S^n \Omega_X) > \frac{1}{2}(n+1)(n+2)$ in Theorem \ref{thm:image-n-cotangent} cannot be dropped. In fact, we will provide examples of surfaces $X$ of general type, with $\Omega_X$ globally generated, such that $h^0(X, \, S^n\Omega_X) =\frac{1}{2}(n+1)(n+2)$ and $X_n$ is the $n{\mathrm{th}}$ Veronese surface for all $n \geq 1$. Thus, no pluri-cotangent map of $X$ is generically finite onto its image. All these surfaces satisfy $c_1^2-c_2=0$.
\medskip

Let $(A, \, M)$ be a polarized abelian threefold, with polarization $M$ of type $(d_1, \, d_2, \, d_3)$, such that there exists a smooth element $X \in |M|$. By the Riemann-Roch Theorem and the ampleness of $M$, we have
\begin{equation} \label{eq:3-fold-polarization}
h^0(A, \, M)=\chi(A, \, M)=\frac{1}{6}M^3= d_1d_2d_3,
\end{equation}
see \cite[Chapter 3]{BL04}. Moreover, by adjunction we get $\omega_X=\mathcal{O}_X(X)$ and so 
\begin{equation} \label{eq:K2-abelian}
K_X^2=M^3=6h^0(A, \, M).
\end{equation}
From the short exact sequence
\begin{equation} \label{eq:short-abelian}
0 \to \mathcal{O}_A \to \mathcal{O}_A(X) \to \omega_X \to 0
\end{equation}
we infer 
\begin{equation}
p_g(X)=h^0(A, \, M)+2, \, \quad q(X)=3,
\end{equation}
hence $\chi(\mathcal{O}_X)=h^0(A, \, M)$. Summing up, $X$ is a minimal surface of general type with $K_X^2=6 \chi(\mathcal{O}_X)$, namely $c_1(X)^2-c_2(X)=0$; moreover the canonical bundle $\Omega_X$ is globally generated  (cf. Example \ref{ex:Omega_generated}) and so $K_X$ is ample (Lemma \ref{lem:ampleness of K_X}).

\begin{proposition} \label{prop:symmetric-power}
For all $n \geq 1$, we have 
\begin{equation} \label{eq:symmetric-power}
H^0(X, \, S^n \Omega_X)=S^n H^0(X, \, \Omega_X) \simeq 	\mathbb{C}^{\frac{(n+1)(n+2)}{2}}.
\end{equation}
Furthermore, the image $X_n$ of $\psi_n \colon \PP(\Omega_X) \to \PP(H^0(X, \, S^n \Omega_X))$ is projectively equivalent to the $n{\emph{th}}$ Veronese surface $\nu_n(\PP^2) \subset \PP^{\frac{n(n+3)}{2}}$.
\end{proposition}
\begin{proof}
Since $\Omega_X$ is globally generated and $h^0(X, \, \Omega_X)=3$, we
have a short exact sequence
\begin{equation} \label{eq:seq-ex1-1}
0 \to \mO_X(-K_X) \to H^0(X, \, \Omega_X) \otimes \mO_X \to \Omega_X \to 0.
\end{equation}
From \cite[p. 577]{Eis94} it follows that \eqref{eq:seq-ex1-1} gives rise to a short exact sequence
\begin{equation} \label{eq:seq-ex1-2}
0 \to S^{n-1} H^0(X, \, \Omega_X) \otimes \mO_X(-K_X) \to S^n H^0(X, \, \Omega_X) \otimes \mO_X \to S^n \Omega_X \to 0,
\end{equation}
where the exactness on the left follows by comparing ranks. The canonical divisor $K_X$ is effective and ample, so $h^0(X, \, -K_X)=h^1(X, \, -K_X)=0$; taking cohomology in \eqref{eq:seq-ex1-2}, we obtain \eqref{eq:symmetric-power}. As a consequence, the map $\mathbb{P}(\sigma_n)$ in diagram \eqref{dia:symmetrization} is an isomorphism for all $n$. Since the $1$-cotangent map $\psi_1 \colon \mathbb{P}(\Omega_X) \to \mathbb{P}(H^0(X, \, \Omega_X)) \simeq \mathbb{P}^2$ is surjective, the image of $\psi_n$ must coincide, up to a projective transformation, with the image of $\nu_n$.
\end{proof}

\begin{remark} \label{rmk:Debarre-restriction}
Another way to state \eqref{eq:symmetric-power} is saying that the restriction map
\begin{equation} \label{eq:restriction-isomorphism}
H^0(A, \, S^n \Omega_A) \to H^0(X, \, S^n \Omega_X)
\end{equation}
is an isomorphism for all $n \geq 1$, cf. \cite[Proposition 13]{Deb05}. 
 
 \end{remark}

\begin{remark} \label{rmk:symmetric-product-genus-3}
When $M$ is a principal polarization, namely $(d_1, \, d_2, \, d_3)=(1, \, 1, \, 1)$, the surface $X$ is a smooth theta divisor of $A$ and we get $p_g(X)=q(X)=3$, $K_X^2=6$. In this case, $X$ is isomorphic to the second symmetric product $\operatorname{Sym}^2C$, where $C$ is a smooth, non-hyperelliptic curve of genus $3$, see \cite[p. 304]{CaCiML98}. 
\end{remark}

The next result shows that, if one assumes that the Albanese image is smooth, then the previous counterexamples are the only ones up to finite \'{e}tale covers. Recall that an abelian cover is a Galois cover with abelian Galois group.

\begin{proposition} \label{prop:all-pluricotangent-dim-2}
If $Y$ is a smooth surface of general type with smooth Albanese image, then the following are equivalent.
\begin{itemize}
\item[$\boldsymbol{(1)}$] $\Omega_Y$ is globally generated and the pluri-cotangent image $Y_n$ has dimension $2$ for all $n\geq 1$.
\item[$\boldsymbol{(2)}$] $\Omega_Y$ is globally generated and the $1$-cotangent image $Y_1$ has dimension $2$.
\item[$\boldsymbol{(3)}$] $Y$ is a finite, \'{e}tale cover of a smooth, ample divisor in an abelian threefold.
\item[$\boldsymbol{(4)}$] $Y$ is a finite, \'{e}tale abelian cover of a smooth, ample divisor in an abelian threefold.
\end{itemize}
\end{proposition}
\begin{proof}
$\boldsymbol{(1)} \Longrightarrow \boldsymbol{(2)}$ Obvious. \\ \\
$\boldsymbol{(2)} \Longrightarrow \boldsymbol{(3)}$ If $\boldsymbol{(2)}$ holds, then by \cite[Proposition 2.14]{Rou09} we have $q(Y)=3$, hence $A:=\operatorname{Alb}(Y)$ is an abelian threefold and the Albanese map $a_Y \colon Y \to A$ is a local immersion onto its smooth image $X \subset A$. This implies that $Y \to X$ is a local bihomomorphism between compact complex manifolds, hence an unramified analytic cover (\cite[Problem 11-9 p. 303]{Lee11}), which is actually an algebraic cover by GAGA. By adjunction, the surface $X$ satisfies $0 < K_X^2=X^3$, so the divisor $X$ is ample in $A$ by \cite[Proposition 4.5.2]{BL04}.  \\ \\
$\boldsymbol{(3)} \Longrightarrow \boldsymbol{(4)}$ Let $X \subset A$ be a smooth, ample divisor in an abelian threefold and let $f \colon Y \to X$ be a finite, \'{e}tale cover. By Lefschetz hyperplane theorem \cite[Theorem 3.1.21]{Laz04} it follows $\pi_1(X) = \pi_1(A)=\mathbb{Z}^6$; thus, since the fundamental group of $X$ is abelian, the cover  $f \colon Y \to X$ is Galois, with abelian Galois group. \\ \\
$\boldsymbol{(4)} \Longrightarrow \boldsymbol{(1)}$ Let $X \subset A$ be a smooth ample divisor in an abelian threefold and let $f \colon Y \to X$ be a finite, \'{e}tale abelian cover. By \cite[p. 200]{Par91}, there exist non-trivial torsion divisors $L_1, \ldots, L_s \in \operatorname{Pic}^0(X)$ such that 
\begin{equation} \label{eq:Pardini-splitting}
f_* \mathcal{O}_Y=\mathcal{O}_X \oplus L_1 \oplus \cdots \oplus L_s.
\end{equation}
Since the cover is \'{e}tale, we have $S^n \Omega_{Y} = f^* (S^n \Omega_X)$ for all $n \geq 1$. Thus $S^n \Omega_Y$ is globally generated (because $S^n \Omega_X$ is) and, by projection formula, we get
\begin{equation} \label{eq:Pardini-for-cotangent}
f_* S^n \Omega_{Y} = S^n \Omega_X \oplus (S^n \Omega_X \otimes L_1) \oplus \cdots \oplus (S^n \Omega_X \otimes L_s).
\end{equation}
Since the divisor $L_j$ is not effective and $K_X-L_j$ is ample, we get
\begin{equation}
H^0(X, \, L_j)=H^1(X, -K_X+L_j)=0.
\end{equation}
Thus, tensoring \eqref{eq:seq-ex1-2} with $L_j$ and passing to cohomology, we deduce $H^0(X, \, S^n \Omega_X \otimes L_j)=0$ for all $j \in \{1, \ldots, s\}$. Hence \eqref{eq:Pardini-for-cotangent} yields $H^0(Y, \, S^n \Omega_{Y})= f^*H^0(X, \, S^n \Omega_{X})$, which in turn implies $Y_n=X_n$. By Proposition \ref{prop:symmetric-power} it follows that $Y_n$ has dimension $2$ for all $n \geq 1$.
\end{proof}

\begin{remark} \label{rmk:order-Gauss-maps}
The argument in the last part of the proof of Proposition \ref{prop:all-pluricotangent-dim-2} can be also used in order to construct examples of surfaces of general type having Gauss maps of arbitrarily large degree. For the sake of simplicity, let us just consider finite, \'{e}tale cyclic covers. Let $X \subset A$ be a smooth ample divisor in an abelian threefold and let $L$ be a non-trivial element of order $k$ in $\mathrm{Pic}^0(X)$. These data define an \'{e}tale $\mathbb{Z}_k$-cover $f \colon Y \to X$, where 
\begin{equation} \label{eq:k-cover}
f_* \mathcal{O}_{Y} = \mathcal{O}_X \oplus L \oplus \cdots \oplus L^{k-1}.
\end{equation}  
As before, we get $H^0(Y, \, S^n \Omega_{Y})= f^*H^0(X, \, S^n \Omega_{X})$, which in turn implies
\begin{equation*}
 \mathbb{G}(n, \, \mathbb{P}(H^0(Y, \, S^n \Omega_{Y})))= \mathbb{G}(n, \, \mathbb{P}(H^0(X, \, S^n \Omega_{X}))). 
\end{equation*}
This shows that, for all $n \geq 1$, the $n$th Gauss map of $Y$ factors through the degree $k$  cover $f \colon Y \to X$, and so its degree is a multiple of $k$. Since $k$ is arbitrary, this construction provides smooth, minimal surfaces of general type $Y$, all of whose Gauss maps have arbitrarily large degree.  
\end{remark}

\section{Open problems} \label{sec:open_problems}
We end the paper with a couple of open problems.

\begin{Op1}
Are there examples of minimal surfaces of general type such that $S^n \Omega_X$ is globally generated for some $n$ and $h^0(X, \, S^n \Omega_X) < \frac{1}{2}(n+1)(n+2)$ holds? If such examples exist, what is the behaviour of the pluri-cotangent map $\psi_n$?
\end{Op1}
This question what asked by the first Author in the MathOverflow thread \href{https://mathoverflow.net/questions/430570/}{MO430570}, without any answer so far. It is motivated by the fact that, in all the examples that we are able to compute, if $S^n \Omega_X$ is globally generated then  $h^0(X, \, S^n \Omega_X) \geq \frac{1}{2}(n+1)(n+2)$. The equality is attained, for instance, by finite \'{e}tale covers of smooth ample divisors in abelian threefolds,  see Subsection \ref{subsec:lambda-6-sharp}.

\begin{Op2}
How should one modify Proposition $\operatorname{\ref{prop:all-pluricotangent-dim-2}}$ if one removes the smoothness assumption for the Albanese image of $Y$? 
\end{Op2}


\bigskip
\bigskip

Francesco Polizzi \\
 Dipartimento di Matematica e Applicazioni \\
Universit\`{a} degli Studi di Napoli ``Federico II" \\
Via Cintia, Monte S. Angelo \\
I-80126 Napoli, Italy. \\
\emph{E-mail address:} \verb|francesco.polizzi@unina.it|

\bigskip

Xavier Roulleau\\
\vspace{0.1cm}
Laboratoire angevin de recherche en math\'ematiques \\
LAREMA, UMR 6093 du CNRS \\ UNIV. Angers, SFR MathStic, 2 Bd Lavoisier \\
49045 Angers Cedex 01, France.
\\
\emph{E-mail address:} \verb|Xavier.Roulleau@univ-angers.fr|

\end{document}